\documentclass[twoside,leqno,twocolumn]{article}
\usepackage{ltexpprt}
\usepackage{epsf,epsfig}
\usepackage{amsmath,amssymb}
\usepackage{mathrsfs}
\usepackage[normalem]{ulem}
\usepackage{color}

\begin{document}

\title{\Large Multilinear Time Invariant System Theory \thanks{Supported by Air Force Office of Scientific Research, the National Science Foundation, and the Lifelong Learning Machines program from DARPA/MTO.}}
\author{Can Chen\thanks{Department of Mathematics and Department of Electrical Engineering and Computer Science, University of Michigan.}\\
\and
Amit Surana\thanks{United Technologies Research Center.}\\
\and
Anthony Bloch\thanks{Department of Mathematics, University of Michigan.}\\
\and
Indika Rajapakse\thanks{Department of Computational Medicine \& Bioinformatics, Medical School and Department of Mathematics, University of Michigan.}
}
\date{}

\maketitle

% Copyright Statement
% When submitting your final paper to a SIAM proceedings, it is requested that you include
% the appropriate copyright in the footer of the paper.  The copyright added should be
% consistent with the copyright selected on the copyright form submitted with the paper.
% Please note that "20XX" should be changed to the year of the meeting.

% Default Copyright Statement
%\fancyfoot[R]{\scriptsize{Copyright \textcopyright\ 2019 by SIAM\\
%Unauthorized reproduction of this article is prohibited}}

% Depending on which copyright you agree to when you sign the copyright form, the copyright
% can be changed to one of the following after commenting out the default copyright statement
% above.

%\fancyfoot[R]{\scriptsize{Copyright \textcopyright\ 20XX\\
%Copyright for this paper is retained by authors}}

%\fancyfoot[R]{\scriptsize{Copyright \textcopyright\ 20XX\\
%Copyright retained by principal author's organization}}

%\pagenumbering{arabic}
%\setcounter{page}{1}%Leave this line commented out.

\begin{abstract} \small\baselineskip=9pt In biological and engineering systems, structure, function and dynamics are highly coupled. Such interactions can be naturally and compactly captured via tensor based state space dynamic representations. However, such representations are not amenable to the standard system and controls framework which requires the state to be in the form of a vector. In order to address this limitation, recently a new class of multiway dynamical systems has been introduced in which the states, inputs and outputs are tensors. We propose a new form of multilinear time invariant (MLTI) systems based on the Einstein product and even-order paired tensors. We extend classical linear time invariant (LTI) system notions including stability, reachability and observability for the new MLTI system representation by leveraging recent advances in tensor algebra.\end{abstract}

\section{Introduction} In many complex systems, such as those arising in biology, capturing the interplay of function, structure and dynamics is critical in order to characterize the underlying mechanisms \cite{7798500}. The human genome is a beautiful example of a multiway dynamical system \cite{Rajapakse_2011}.  The organization of the interphase nucleus reflects a dynamical interaction between 3D genome structure, function, and its relationship to phenotype, a concept known as the 4D Nucleome (4DN) \cite{Chen_2015}. 4DN research requires a comprehensive view of genome-wide structure, gene expression, the proteome, and phenotype which fits naturally with a tensorial representation \cite{7798500}.
The mathematical foundation of tensor based representation and analysis could play a critical role in the study of the human genome as well as social networks, cognitive science, signal processing and machine learning.
%The mathematical foundation of tensor based representation and analysis could play a critical role in not only study of the human genome, but in general for the study of complex systems arising in many domains involving tensors, such as social networks, cognitive science, signal processing, and machine learning.

The notion of multilinear dynamical system or multilinear time invariant (MLTI) system was first introduced by Rogers et al. \cite{rogers_2013} for modeling of tensor time series, and Surana et al. \cite{7798500} built the model by using tensor Tucker products to capture the evolution of the multilinear dynamics. Compared to classical linear time invariant (LTI) based approaches which fit vector or matrix models to tensor time series, MLTI representation provides a more natural, compact and accurate representation of tensorial data with fewer model parameters. By using tensor unfolding, an operation that transforms a tensor into a matrix, Rogers et al. \cite{rogers_2013} and Surana et al. \cite{7798500} developed methods for model identification/reduction from tensor time series data, and demonstrated benefits of the MTLI representation compared to the classical LTI approach. However, this representation of MLTI systems is limited by the fact that it assumes the multilinear operators are formed from the Tucker products of matrices, and thus precludes more general tensorial representation. Moreover, while tensor unfolding enables one to transform MLTI into a classical LTI representation for computational purposes (i.e. enables use of matrix algebra), one loses the inherent tensor algebraic structure which otherwise could be exploited to develop system theoretic concepts.
%\textcolor{red}{Tensor calculus was originally developed around the later nineteenth century \cite{Ricci_1892}}.

In this paper, we propose a more generalized MLTI system model using the Einstein product, from which the Tucker product based MTLI representation can be obtained as a special case. The Einstein product is a tensor contraction operation used quite often in tensor calculus and has profound applications in the study of continuum mechanics and the field of relativity theory \cite{Lai_2009,Einstein_2007}. The proposed generalized MLTI system model takes a very similar form to the classical LTI system  model, and is thus more naturally suited to develop system theoretic concepts. Moreover, the space of even-order tensors equipped with the Einstein product has many desirable properties. Brazell et al. \cite{doi:10.1137/100804577} in 2013 discovered that one particular tensor unfolding gives rise to an isomorphism from this tensor space (of even-order tensors equipped with the Einstein  product) to the general linear group, i.e. group of invertible matrices.

Building on the results of Brazell et al. \cite{doi:10.1137/100804577} and the notion of block tensors \cite{doi:10.1137/110820609}, we propose new tensor notions for positive definiteness, unfolding rank and a new way of concatenation of tensors to create block tensors. Using these tensor constructs, we develop tensor algebraic conditions for stability, reachability and observability for the generalized MLTI systems.  Interestingly, these new conditions look analogous to the classical conditions for stability, reachability and observability for LTI systems, and reduce to them in special cases. To the best of the authors' knowledge, MTLI system representation using the Einstein product and formulation of these tensor algebraic system theoretic conditions has never been reported in the literature. Due to space limitations, we only provide proofs for selected results. A more comprehensive publication is under preparation.

The paper is organized in five sections. We start with basics of tensor algebra followed with tensor groups, block tensors and tensor eigenvalue decompositions in Section \ref{sec:TensorAlgebra}. A new general representation of MLTI systems is introduced in Section \ref{sec:MLTI}, and generalization of stability, reachability and observability conditions for the MLTI systems is also discussed. A simple single input single output MLTI system example is given in Section \ref{sec:example}, and we conclude with directions for future research in Section \ref{sec:future}.

\section{Tensor Algebra}\label{sec:TensorAlgebra}
We use concepts and notations for tensor algebra from the comprehensive works of Kolda et al. \cite{Kolda06multilinearoperators,kolda2009tensor} and Ragnarsson et al. \cite{doi:10.1137/110820609}. A \textit{tensor} is a multidimensional array. The \textit{order} of a tensor is the number of its dimensions.  An $N$-th order tensor usually is denoted by $\mathcal{X}\in \mathbb{R}^{J_1\times J_2\times  \dots \times J_N}$. The sets of indexed indices and size of $\mathcal{X}$ are denoted by $\textbf{j}=\{j_1,j_2,\dots,j_N\}$ and $\mathscr{J}=\{J_1,J_2,\dots,J_N\}$, respectively. $|\mathscr{J}|$ represents the product of all elements in $\mathscr{J}$, and $\textbf{j}\in[\mathscr{J}]$ can be interpreted as $j_n=1,2,\dots,J_n$ for $n=1,2,\dots,N$. It is therefore reasonable to consider scalars $x\in\mathbb{R}$ as zero-order tensors, vectors $\textbf{v}\in\mathbb{R}^{J}$ as first-order tensors, and matrices $\textbf{A}\in\mathbb{R}^{J\times I}$ as second-order tensors.

By extending the notion of vector outer product, the \textit{outer product} of two tensors $\mathcal{X}\in \mathbb{R}^{J_1\times J_2\times \dots \times J_N}$ and $\mathcal{Y}\in \mathbb{R}^{I_1\times I_2\times \dots \times I_M}$ is defined as
\begin{equation}
(\mathcal{X}\circ \mathcal{Y})_{j_1j_2\dots j_Ni_1i_2\dots i_M}=\mathcal{X}_{j_1j_2\dots j_N}\mathcal{Y}_{i_1i_2\dots i_M}. \label{eq99}
\end{equation}
In contrast, the \textit{inner product} of two tensors $\mathcal{X},\mathcal{Y}\in \mathbb{R}^{J_1\times J_2\times \dots \times J_N}$ is defined as
\begin{equation}
\langle \mathcal{X},\mathcal{Y}\rangle =\sum_{\textbf{j}=\textbf{1}}^{\mathscr{J}}\mathcal{X}_{j_1j_2\dots j_N}\mathcal{Y}_{j_1j_2\dots j_N}\label{eq3}
\end{equation}
leading to the \textit{Frobenius norm} $\|\mathcal{X}\|^2=\langle \mathcal{X},\mathcal{X}\rangle$. The notation $\sum_{\textbf{j}=\textbf{1}}^{\mathscr{J}}$ can be read as an abbreviation of $N$  summations over all indices $j_n=1,2,\dots,J_n$ for $n=1,2,\dots,N$.

The \textit{matrix tensor multiplication} $\mathcal{X} \times_{n} \textbf{A}$ along mode $n$ for a matrix $\textbf{A}\in  \mathbb{R}^{I\times J_n}$ is defined by
\begin{equation}
(\mathcal{X} \times_{n} \textbf{A})_{j_1j_2\dots j_{n-1}ij_{n+1}\dots j_N}=\sum_{j_n=1}^{J_n}\mathcal{X}_{j_1j_2\dots j_n\dots j_N}\textbf{A}_{ij_n}.\label{eq4}
\end{equation}
This product can be generalized to what is known as the \textit{Tucker product},
\begin{equation}
\begin{split}
&\mathcal{X}\times_1 \textbf{A}_1 \times_2\dots \times_{N}\textbf{A}_N\\&=\mathcal{X}\times\{\textbf{A}_1,\textbf{A}_2\dots,\textbf{A}_N\}\in  \mathbb{R}^{I_1\times I_2\times\dots \times I_N},\label{eq5}
\end{split}
\end{equation}
where, $\textbf{A}_n\in \mathbb{R}^{I_n\times J_n}$. If an $N$-th order tensor $\mathcal{Y}\in  \mathbb{R}^{I_1\times I_2\times\dots \times I_N}$ can be expressed as $\mathcal{Y} = \mathcal{X}\times\{\textbf{A}_1,\textbf{A}_2\dots,\textbf{A}_N\}$, the decomposition is referred to as the \textit{Tucker decomposition}. When the  \textit{core tensor} $\mathcal{X}$ possesses the ``higher-order diagonal'' property and the \textit{factor matrices} $\textbf{A}_{n},n=1,\cdots,N$ are unitary, it is also called the \textit{Higher-Order Singular Value Decomposition} (HOSVD), a multilinear generalization of the matrix Singular Value Decomposition (SVD) \cite{Lathauwer2000}.

\textit{Tensor unfolding} is considered as a critical operation in tensor computations \cite{doi:10.1137/110820609, Kolda06multilinearoperators,kolda2009tensor}. In order to unfold a tensor $\mathcal{X}\in\mathbb{R}^{J_1\times J_2\times\dots \times J_N}$ into a matrix, we use an index mapping function $ivec(\cdot,\mathscr{J}):\underbrace{\mathbb{Z}^+\times \mathbb{Z}^+\times \dots \times \mathbb{Z}^+}_{N }\rightarrow \mathbb{Z}^+$ defined by Ragnarsson et al. \cite{doi:10.1137/110820609}, which is given as
\begin{equation}
ivec(\textbf{j},\mathscr{J}) = j_1+\sum_{k=2}^N(j_k-1)\prod_{l=1}^{k-1}J_l.
\end{equation}
%It is necessary to select an integer $z$ such that $1\leq z\leq N$ and a permutation $\mathbb{S}$ of $1$ to $N$. Suppose that $\textbf{r} = \mathbb{S}(1:z)$ and $\textbf{c} = \mathbb{S}(z+1:N)$.
Suppose that $\textbf{r} = \mathbb{S}(1:z)$ and $\textbf{c} = \mathbb{S}(z+1:N)$, where $z$ is an integer such that $1\leq z < N$, and $\mathbb{S}$ is a vector from the set of all permutations of $1$ to $N$. Define $\mathscr{J}(\textbf{r}) = \{J_{\mathbb{S}(1)},J_{\mathbb{S}(2)},\dots,J_{\mathbb{S}(z)}\}$ and $\mathscr{J}(\textbf{c}) = \{J_{\mathbb{S}(z+1)},J_{\mathbb{S}(z+2)},\dots,J_{\mathbb{S}(N)}\}$. Then the $\textbf{r}\times \textbf{c}$ unfolding matrix of $\mathcal{X}$ denoted by $\textbf{X}_{\textbf{r}\times \textbf{c}}$ is given by
\begin{equation}
\textbf{X}_{\textbf{r}\times \textbf{c}} (j,i) = \mathcal{X}^{\mathbb{S}}_{j_1j_2\dots j_zi_1i_2\dots i_{N-z}},\label{eq20}
\end{equation}
where, $j = ivec(\textbf{j},\mathscr{J}(\textbf{r}))$, $i = ivec(\textbf{i},\mathscr{J}(\textbf{c}))$ and $\mathcal{X}^{\mathbb{S}}$ is the $\mathbb{S}$-transpose of $\mathcal{X}$ (see (2.9) in \cite{doi:10.1137/110820609}). In particular, when $z=1$ and $\mathbb{S} = \{n,1:n-1,n+1:N\}$, the tensor unfolding is called the \textit{n-mode matricization}.

\subsection{Einstein Product and Isomorphism}

Here we discuss the notion of \textit{even-order paired tensors} and the \textit{Einstein product} which will play an important role in developing the MLTI systems theory. Similarly, these notions proved to be useful in numerical solutions
of master equations associated with Markov processes on extremely large state spaces \cite{tensortrain}. Even-order paired tensors were originally proposed by Huang and Qi \cite{2017arXiv170504315H} in the context of elasticity tensors in solid mechanics. It turns out that compared to even-order non-paired tensors, the even-order paired tensors are easier for bookkeeping and can be conveniently manipulated using tensor algebra for MLTI systems, see Section \ref{sec:MLTI}.

\begin{Definition}
Given an even-order tensor $\mathcal{A}\in\mathbb{R}^{J_1\times I_1\times \dots \times J_N\times I_N}$, if its indices can be divided into $N$ adjacent blocks $\{j_1i_1\},\dots, \{j_Ni_N\}$, then  $\mathcal{A}$ is called an even-order paired tensor.
\end{Definition}

\begin{Definition}
Given two even-order paired tensors $\mathcal{A}\in\mathbb{R}^{J_1\times K_1\times \dots J_N\times K_N}$ and $\mathcal{B}\in\mathbb{R}^{K_1\times I_1 \times \dots \times K_N\times I_N}$, the \textit{Einstein product} $\mathcal{A}*\mathcal{B}\in \mathbb{R}^{J_1\times I_1 \times \dots \times J_N\times I_N}$ is defined by
\begin{equation}
(\mathcal{A}* \mathcal{B})_{j_1i_1\dots j_N i_N}=\sum_{\textbf{k}= \textbf{1}}^{\mathscr{K}} \mathcal{A}_{j_1k_1\dots j_Nk_N}\mathcal{B}_{k_1i_1\dots k_Ni_N}.\label{eq10}
\end{equation}
If $\mathcal{B}\in\mathbb{R}^{K_1\times K_2\times\dots \times K_N}$, the Einstein product is still valid  by treating $\mathscr{I} = \textbf{1}$ in (\ref{eq10}). Note that the notion of Einstein product is not restricted to even-order paired tensors and can be defined more generally.
\end{Definition}

The above Einstein product computes the summation of two even-order paired tensors over alternating indices. Brazell et al. \cite{doi:10.1137/100804577} investigated properties for even-order non-paired tensors $\mathcal{A}\in\mathbb{R}^{J_1\times\dots \times J_N\times I_1\times\dots\times I_N}$ under the Einstein product through construction of an isomorphism to GL($\mathbb{R}$) (general linear group). The existence of the isomorphism enables one to generalize several matrix concepts, such as orthogonality, invertibility and eigenvalue decomposition to the tensor case \cite{doi:10.1137/100804577,doi.org/10.1016,doi:10.1080/03081087.2018.1500993,doi:10.1080/03081087.2015.1071311,doi:10.1080/03081087.2015.1083933}. We establish an analogous isomorphism  for even-order paired tensors by a permutation of indices.
\begin{Definition}
Define the following transformation $\varphi$: $\mathbb{T}_{J_1I_1\dots J_NI_N}(\mathbb{R}) \rightarrow \mathbb{M}_{|\mathscr{J}|,|\mathscr{I}|}(\mathbb{R})$ with $\varphi(\mathcal{A})=\textbf{A}$ defined component-wise as
\begin{equation}
\mathcal{A}_{j_1i_1\dots j_Ni_N}\xrightarrow{\varphi} \textbf{A}_{ivec(\textbf{j},\mathscr{J})ivec(\textbf{i},\mathscr{I})},\label{eq:12}
\end{equation}
where, $\mathbb{T}_{J_1I_1\dots J_NI_N}(\mathbb{R})$ is the set of all $J_1\times I_1\times \dots \times J_N\times I_N$ even-order paired tensors and $\mathbb{M}_{|\mathscr{J}|,|\mathscr{I}|}(\mathbb{R})$ is the set of all $|\mathscr{J}|\times |\mathscr{I}|$ matrices.
\end{Definition}

The transformation $\varphi$ can be viewed as a tensor unfolding discussed in (\ref{eq20}) with $z= N$ and $\mathbb{S} = \{1,3,\dots,2N-1,2,4,\dots,2N\}$. Brazell et al. prove that $\varphi$ is an isomorphism for fourth-order non-paired tensors, and we extend the key results (Corollary 3.3 in \cite{doi:10.1137/100804577}) for even-order paired tensors of any order.

\begin{corollary}
Suppose that $J_n=I_n$ for all $n$ and $\mathbb{M}_{|\mathscr{J}|,|\mathscr{I}|}(\mathbb{R})=\text{GL}(|\mathscr{J}|,\mathbb{R})$. $\mathbb{T}_{J_1J_1\dots J_NJ_N}(\mathbb{R})$ is a group equipped with the Einstein product, and $\varphi$ is a group isomorphism. Moreover, $\mathbb{T}_{J_1J_1\dots J_NJ_N}(\mathbb{R})$ also forms a tensor ring under addition and the Einstein product, and $\varphi$ is a ring isomorphism.
\end{corollary}

For an even-order paired tensor $\mathcal{A} \in \mathbb{R}^{J_1\times I_1 \times \dots \times J_N\times I_N}$,  $\mathcal{T} \in \mathbb{R}^{I_1\times J_1 \times \dots \times I_N\times J_N}$ is called the \textit{U-transpose} of $\mathcal{A}$ if $\mathcal{T}_{i_1j_1\dots i_Nj_N} = \mathcal{A}_{j_1i_1\dots j_Ni_N}$ and is denoted by $\mathcal{A}^{\top}$. We refer to an even-order paired tensor that is identical to its U-transpose as \textit{weakly symmetric}. An even-order ``square'' tensor $\mathcal{D} \in \mathbb{R}^{J_1\times J_1 \times \dots \times J_N\times J_N}$ is called the \textit{U-diagonal} tensor if all its entries are zeros except for $\mathcal{D}_{j_1j_1\dots j_Nj_N}$. In particular, if all the diagonal entires $\mathcal{D}_{j_1j_1\dots j_Nj_N}=1$, then $\mathcal{D}$ is the \textit{U-identity} tensor, denoted by $\mathcal{I}$. An even-order square tensor $\mathcal{U}\in  \mathbb{R}^{J_1\times J_1 \times \dots \times J_N\times J_N}$ is \textit{U-orthogonal} if $\mathcal{U}*\mathcal{U}^{\top}=\mathcal{U}^{\top}*\mathcal{U}=\mathcal{I}$. Furthermore, for an even-order square tensor $\mathcal{A}\in \mathbb{R}^{J_1\times J_1 \times \dots \times J_N\times J_N}$, if there exists a tensor $\mathcal{X}\in\mathbb{R}^{J_1\times J_1 \times \dots \times J_N\times J_N}$ such that $\mathcal{A}*\mathcal{X}=\mathcal{X}*\mathcal{A} = \mathcal{I}$, then $\mathcal{X}$ is called the \textit{U-inverse} of $\mathcal{A}$, denoted by $\mathcal{A}^{-1}$. ``U'' stands for ``unfolding'' in all the definitions. Besides the properties discussed above, we can define \textit{U-positive definiteness} for even-order square tensors, which is similarly discussed in \cite{tensortrain}.
\begin{Definition}
An even-order square tensor $\mathcal{A} \in \mathbb{R}^{J_1\times J_1 \times \dots \times J_N\times J_N}$ is U-positive definite if its corresponding homogeneous polynomial
\begin{equation}
h(\mathcal{X})=\mathcal{X}^{\top}*\mathcal{A}*\mathcal{X} > 0 \label{eq13}
\end{equation}
for all $\mathcal{X}\neq\mathcal{O} \in \mathbb{R}^{J_1\times J_2\times\dots \times J_N}$, where $\mathcal{O}$ denotes the zero tensor.
\end{Definition}

It is straightforward to show that an even-order square tensor $\mathcal{A}$ is U-positive definite if and only if $\varphi(\mathcal{A})$ is positive definite. Moreover, U-positive definiteness implies U-invertibility of even-order square tensors from the isomorphism property.

The notions of linear dependence and independence for tensor spaces are defined in a similar way to those for  vector spaces. Let $\mathbb{T}_{J_1J_2\dots J_N}(\mathbb{R})$ be a set of all tensors $\mathcal{X}\in \mathbb{R}^{J_1\times J_2\times  \dots \times J_N}$. A basis $\mathscr{B}$ of the tensor space $\mathbb{T}_{J_1J_2\dots J_N}(\mathbb{R})$ is a linearly independent subset of $\mathbb{T}_{J_1J_2\dots J_N}(\mathbb{R})$ that spans $\mathbb{T}_{J_1J_2\dots J_N}(\mathbb{R})$. Additionally, $\text{dim} \big (\mathbb{T}_{J_1J_2\dots J_N}(\mathbb{R})\big)$ is equal to the cardinality of the basis $\mathscr{B}$. Ji et al. \cite{doi.org/10.1016} propose the null space $N(\mathcal{A})$ and the range $R(\mathcal{A})$ of an even-order non-paired tensor $\mathcal{A}$, and establish some elementary properties. Analogous definitions for even-order paired tensors are as follows:

\begin{Definition}
Define the \textit{null space} and \textit{range} of an even-order paired tensor $\mathcal{A}\in \mathbb{R}^{J_1\times I_1\times \dots \times J_N\times I_N}$ to be
\begin{equation}
\begin{split}
N(\mathcal{A})&=\{\mathcal{X}\in \mathbb{R}^{I_1\times I_2\times \dots \times I_N}:\mathcal{A}*\mathcal{X}=\mathcal{O}\}, \\R(\mathcal{A})&=\{\mathcal{A}*\mathcal{X}:\mathcal{X}\in \mathbb{R}^{I_1\times I_2\times \dots \times I_N}\},
\end{split}
\end{equation}
respectively. Moreover, define nullity$_U(\mathcal{A}$)=dim\big($N(\mathcal{A})$\big) and rank$_U(\mathcal{A}$)=dim\big($R(\mathcal{A})$\big).
\end{Definition}

We find that above notion of rank$_U(\mathcal{A}$) is equivalent to the \textit{unfolding rank} defined by Liang et al. \cite{doi:10.1080/03081087.2018.1500993}.
%Hence, we also refer to the rank$_U(\mathcal{A}$) as the unfolding rank in this paper.
%\textcolor{red}{The notions nullity$_U(\mathcal{A}$) and rank$_U(\mathcal{A}$) are not discussed by Ji et al. in \cite{doi.org/10.1016}, but we find that rank$_U(\mathcal{A}$) is equivalent to the \textit{unfolding rank} defined by Liang et al. \cite{doi:10.1080/03081087.2018.1500993}. Hence, we also refer rank$_U(\mathcal{A}$) as unfolding rank in this paper.}

\begin{proposition}\label{prop3}
Let  $\mathcal{A} \in \mathbb{R}^{J_1\times I_1 \times \dots \times J_N\times I_N}$ be an even-order paired tensor. Then
\begin{align}
\text{rank}_U(\mathcal{A})=\text{rank}\big(\varphi(\mathcal{A})\big).
\end{align}
\end{proposition}

The proof is based on the definitions of the unfolding rank and the bases in tensor spaces as discussed above. Hence, we also refer to the rank$_U(\mathcal{A}$) as the unfolding rank in this paper. In addition, Liang et al. \cite{doi:10.1080/03081087.2018.1500993} propose the \textit{unfolding determinant}, $\text{det}_U$, for even-order non-paired tensors. Similar definition/results can be extended for even-order paired tensors.

\subsection{Block Tensors}
%The field of block matrix computations has matured and has wide applications in engineering problems, and block tensors also become increasingly significant as block matrices \cite{doi:10.1137/110820609}.
Analogously to  block matrices, one can define the notion of block tensors. The block tensors introduced by Sun et al. \cite{doi:10.1080/03081087.2015.1071311} for even-order non-paired tensors have the limitation of introducing too many zeros into the tensor and increasing its size. For tensors of the same size, we explore a new block tensor construction that does not introduce any wasteful zeros, and thus could offer computational advantage.

\begin{Definition}\label{def:nmodedef}
Let  $\mathcal{A},\mathcal{B}\in\mathbb{R}^{J_1\times I_1\times \dots \times J_N\times I_N}$ be two even-order paired tensors of the same size. Then the \textit{$n$-mode row block tensor} is defined to be $\begin{Vmatrix}
	\mathcal{A} & \mathcal{B}
\end{Vmatrix}_n\in \mathbb{R}^{J_1\times I_1 \times \dots \times J_n\times 2I_n \times\dots \times J_N\times I_N}$ such that
\begin{equation}
(\begin{Vmatrix}
	\mathcal{A} & \mathcal{B}
\end{Vmatrix}_n)_{j_1l_1\dots j_Nl_N}=
\begin{cases}
\mathcal{A}_{j_1l_1\dots j_Nl_N}, \textbf{j}\in[\mathscr{J}], \text{ } \textbf{l}\in[\mathscr{I}]\\
\mathcal{B}_{j_1l_1\dots j_Nl_N}, \textbf{j}\in[\mathscr{J}], \text{ } \textbf{l}\in[\mathscr{L}],
\end{cases}
\end{equation}
where, $\mathscr{L}=\mathscr{I}$ except $l_n=I_n+1,I_n+2,\dots,2I_n$.
\end{Definition}

The \textit{$n$-mode column block tensor} $\begin{Vmatrix}
	\mathcal{A} &
	\mathcal{B}
\end{Vmatrix}_n^{\top}\in \mathbb{R}^{J_1\times I_1 \times \dots \times 2J_n\times I_n \times\dots \times J_N\times I_N}$ can be defined in a similar manner. However, the blocks of even-order paired tensors usually do not map to contiguous blocks in their unfolding  \cite{doi:10.1137/110820609}, and sometimes that may increase the complexity in computations. Ragnarsson et al. \cite{doi:10.1137/110820609} show that there exists a row permutation matrix $\textbf{Q}$ and a column permutation matrix $\textbf{P}$ such that the blocks of a tensor $\mathcal{A}$ can be mapped to contiguous blocks in the unfolding $\textbf{Q}\textbf{A}_{\textbf{r}\times \textbf{c}}\textbf{P}$. The following proposition shows that for $n$-mode row block tensors, only column permutations are required.
\begin{proposition}\label{pro6}
Let  $\mathcal{A},\mathcal{B}\in\mathbb{R}^{J_1\times I_1\times \dots \times J_N\times I_N}$ be two even-order paired tensors. Then
\begin{equation}
\varphi(\begin{Vmatrix}
	\mathcal{A} & \mathcal{B}
\end{Vmatrix}_n) =
\begin{bmatrix}
	\varphi(\mathcal{A}) & \varphi(\mathcal{B})
\end{bmatrix}\textbf{P},
\end{equation}
where, $\textbf{P}$ is a permutation matrix. In particular, when $\mathscr{I} = \textbf{1}$ or $n=N$, $\textbf{P}$ is the identity matrix.
\end{proposition}

The proof follows immediately from the definition of index mapping function $ivec(\textbf{i},\mathscr{I})$ and $n$-mode row block tensors. Based on Ragnarsson et al.'s results, it follows that the blocks of $n$-mode column block tensors map to contiguous blocks in its unfolding up to some row permutations. Moreover, Proposition \ref{pro6} helps us to establish the relation between unfolding rank and matrix rank for block tensors and their unfolding.

\begin{corollary}\label{coro1}
Let $\mathcal{A},\mathcal{B}\in\mathbb{R}^{J_1\times I_1\times \dots \times J_N\times I_N}$ be  two even-order paired tensors. Then
\begin{align}
\text{rank}_U(\begin{Vmatrix}
	\mathcal{A} & \mathcal{B}
\end{Vmatrix}_n) = \text{rank}(\begin{bmatrix}
	\varphi(\mathcal{A}) & \varphi(\mathcal{B})
\end{bmatrix}).
\end{align}
\end{corollary}

Lastly, we generalize the $n$-mode block tensors for multiple blocks. Given $K$ even-order paired tensors $\mathcal{X}_n\in\mathbb{R}^{J_1\times I_1\times \dots\times J_N\times I_N}$, one can apply Definition \ref{def:nmodedef} successively to create a $J_1\times I_1\times \dots \times J_n\times I_nK  \times \dots \times J_N\times I_N$ even-order $n$-mode row block tensor. We define a more general concatenation approach as follows:
%\textcolor{red}{Although there may exist various options by choosing different mode concatenations, the following is superior in many block tensors' applications.}

%\textcolor{red}{
\begin{Definition}
Given $K$ even-order paired tensors $\mathcal{X}_n\in\mathbb{R}^{J_1\times I_1\times \dots\times J_N\times I_N}$, if $K=K_1K_2\dots K_N$, the $J_1\times I_1K_1\times \dots \times J_N\times I_NK_N$ even-order mode row block tensor $\mathcal{Y}$ can be constructed in the following way:
\begin{itemize}
 \item Compute the 1-mode row block tensor concatenation of the sets $\{\mathcal{X}_1,\cdots,\mathcal{X}_{K_1}\}$, $\{\mathcal{X}_{K_1+1},\cdots,\mathcal{X}_{2K_1}\}$ and so on to obtain $K_2K_3\dots K_N$ block tensors denoted by $\mathcal{X}_1^{(1)},\mathcal{X}_2^{(1)},\dots,\mathcal{X}_{K_2K_3\dots K_N}^{(1)}$;
 \item Compute the 2-mode row block tensors concatenation of the sets $\{\mathcal{X}_1^{(1)},\cdots,\mathcal{X}_{K_2}^{(1)}\}$, $\{\mathcal{X}_{K_2+1}^{(1)},\cdots,\mathcal{X}_{2K_2}^{(1)}\}$ and so on to obtain $K_3K_4\dots K_N$ block tensors denoted by $\mathcal{X}_1^{(2)},\mathcal{X}_2^{(2)},\dots,  \mathcal{X}_{K_3K_4\dots K_N}^{(2)}$;
 \item Keep repeating the process until the last $N$-mode row block tensor is obtained.
 \end{itemize}
  We denote the mode row block tensor as $\mathcal{Y} = \begin{Vmatrix} \mathcal{X}_1 &\mathcal{X}_2 & \dots & \mathcal{X}_{K}\end{Vmatrix}$.
\end{Definition}
%}
%\textcolor{blue}{Tony: Check where I replaced "every" by "all" in the above --I think correct}

The generalized \textit{mode column block tensors} with multiple blocks can be constructed in a similar manner. Both Proposition \ref{pro6} and Corollary \ref{coro1} hold for the general mode block tensor construction defined above. In particular, when $\mathscr{I}=\textbf{1}$, the above generalized mode row block tensor maps exactly to contiguous blocks in its unfolding under $\varphi$ which could be beneficial in many block tensor applications, e.g. see next subsection.

\subsection{Tensor Eigenvalue Decomposition}
 Eigenvalue problems for higher-order tensors were first explored by Qi \cite{QI20051302} and Lim \cite{singularvaluetensor} independently in 2005. Brazell et al. \cite{doi:10.1137/100804577} formulated a new tensor eigenvalue problem through the isomorphism $\varphi$ for fourth-order non-paired tensors, and  Cui et al. \cite{doi:10.1080/03081087.2015.1083933} extended the result to general even-order tensors. Note that Tensor Eigenvalue Decomposition (TEVD) derived from the unfolding  transformation is distinct from other notions of tensor decompositions, such as the Candecomp/Parafac decomposition (CP decomposition) or the Tucker decomposition. We exploit our new notion of mode block tensors to express the full tensor eigenvalue decomposition of an even-order paired tensor.
\begin{Definition}
Let $\mathcal{A}\in\mathbb{R}^{J_1\times J_1\times \dots \times J_N\times J_N}$ be an even-order square tensor. If $\mathcal{X}\in\mathbb{C}^{J_1\times J_2\dots\times J_N}$ is a nonzero $N$-th order tensor, $\lambda\in \mathbb{C}$, and $\mathcal{X}$ and $\lambda$ satisfy
\begin{equation}
\mathcal{A}*\mathcal{X}=\lambda\mathcal{X},
\end{equation}
then we call $\lambda$ and $\mathcal{X}$ as the U-eigenvalue and U-eigentensor of $\mathcal{A}$, respectively. Moreover, the tensor eigenvalue decomposition is given by
\begin{equation}
\mathcal{A} = \mathcal{V} *\mathcal{D}* \mathcal{V}^{-1} \label{eq216},
\end{equation}
where, $\mathcal{D}$ is an U-diagonal tensor with U-eigenvalues on its diagonal, and $\mathcal{V}\in\mathbb{R}^{J_1\times J_1\times \dots \times J_N\times J_N}$ is a mode row block tensors consisting of all the U-eigentensors, i.e.
$\mathcal{V} = \begin{Vmatrix} \mathcal{X}_1 & \mathcal{X}_2&\dots & \mathcal{X}_{|\mathscr{J}|}\end{Vmatrix}$. Here we have chosen $K_1=J_1,\cdots,K_n=J_n$ in applying the mode row block tensor operation which enables to express the TEVD in the form (\ref{eq216}) analogous to the matrix case.
\end{Definition}

The algebraic and geometric multiplicity of tensor eigenvalues can be defined as for matrices. The following provides a generalization of the Caley-Hamilton theorem for the tensor case.
\begin{lemma}\label{lem3}
If $\mathcal{A} \in\mathbb{R}^{J_1\times J_1\times \dots \times J_N\times J_N}$ is an even-order square tensor, then $\mathcal{A}$ satisfies its own characteristic polynomial $p(\lambda) = \text{det}_U(\lambda \mathcal{I}-\mathcal{A})$, i.e. $p(\mathcal{A}) = \mathcal{O}$.
\end{lemma}
\begin{proof}
The proof follows immediately by applying $\varphi$.
\end{proof}

\section{MLTI Systems Theory}\label{sec:MLTI}
In order to describe the evolution of tensor time series, the authors in \cite{7798500,rogers_2013} introduced a MLTI system involving the Tucker product as follows,
%The Tucker product based representation of MLTI systems as introduced in  is given by
\begin{align}
\begin{cases}
\mathcal{X}_{t+1}&=\mathcal{X}_{t}\times\{\textbf{A}_1,\dots,\textbf{A}_N\}+\mathcal{U}_t\times\{\textbf{B}_1,\dots,\textbf{B}_N\}\\
\mathcal{Y}_{t}&=\mathcal{X}_{t}\times\{\textbf{C}_1,\dots,\textbf{C}_N\}
\end{cases},\label{eq9}
\end{align}
where, $\mathcal{X}_{t}\in\mathbb{R}^{J_1\times J_2\times \dots \times J_N}$ is the latent state space tensor, $\mathcal{Y}_{t}\in\mathbb{R}^{I_1\times I_2\times \dots \times I_N}$ is the output tensor, and $\mathcal{U}_t\in \mathbb{R}^{K_1\times K_2\times \dots \times K_N}$ is a control tensor. $\textbf{A}_n\in \mathbb{R}^{J_n\times J_n}$, $\textbf{B}_n\in\mathbb{R}^{J_n\times K_n}$ and $\textbf{C}_n\in \mathbb{R}^{I_n\times J_n}$ are real valued matrices for $n=1,2,\dots,N$. The Tucker product provides a suitable way to deal with MLTI systems because it allows one to exploit matrix computations. In particular, using the Kronecker product, one can transform the system (\ref{eq9}) into a standard LTI system \cite{rogers_2013} and then apply the standard LTI systems concepts for analysis. However, this representation is limited by several factors. Firstly, the set of multilinear operators resulting from the component matrices $\textbf{A}_n$, $\textbf{B}_n$ and $\textbf{C}_n$ consists of a special case and does not capture the  more general multilinear evolution of tensor dynamics (see system (\ref{eq11})). Secondly, once transformed into an LTI system via the Kronecker product, there is no unique way to recover the original tensor based representation. Thus, one loses the inherent tensor algebraic structure which otherwise could be exploited to develop system theoretic concepts such as reachability and observability Gramians more naturally. 

We find that (\ref{eq9}) can be replaced by a more general representation using the notion of even-order paired tensors and the Einstein product.
\begin{Definition}
A more general representation of MLTI system is given by
\begin{align}
\begin{cases}
\mathcal{X}_{t+1}&=\mathcal{A}*\mathcal{X}_{t}+\mathcal{B}*\mathcal{U}_{t}\\
\mathcal{Y}_{t}&=\mathcal{C}*\mathcal{X}_{t}
\end{cases},\label{eq11}
\end{align}
where, $\mathcal{A}\in\mathbb{R}^{J_1\times J_1\times \dots \times J_N\times J_N}$, $\mathcal{B}\in\mathbb{R}^{J_1\times K_1\times \dots \times J_N\times K_N}$ and $\mathcal{C}\in\mathbb{R}^{I_1\times J_1\times \dots \times I_N\times J_N}$ are even-order paired tensors.
\end{Definition}

\begin{proposition}
The governing equations (\ref{eq11}) can be obtained from (\ref{eq9}) by setting $\mathcal{A}$, $\mathcal{B}$ and $\mathcal{C}$ to be the outer products of component matrices $\{\textbf{A}_1,\textbf{A}_2,\dots,\textbf{A}_N\}$, $\{\textbf{B}_1,\textbf{B}_2,\dots,\textbf{B}_N\}$ and $\{\textbf{C}_1,\textbf{C}_2,\dots,\textbf{C}_N\}$ respectively.
\end{proposition}
\begin{proof}
Consider a general Tucker decomposition $\mathcal{Y}=\mathcal{X}\times\{\textbf{A}_1,\textbf{A}_2,\dots,\textbf{A}_N\}$ and rewrite it elementwise, i.e.
\begin{align*}
\mathcal{Y}_{i_1i_2\dots i_N}&=\sum_{\textbf{j}=\textbf{1}}^{\mathscr{J}} (\textbf{A}_1)_{i_1j_1}\dots (\textbf{A}_N)_{i_Nj_N}\mathcal{X}_{j_1j_2\dots j_N}\\
&=\sum_{\textbf{j}=\textbf{1}}^{\mathscr{J}} \mathcal{A}_{i_1j_1\dots i_Nj_N}\mathcal{X}_{j_1j_2\dots j_N},
\end{align*}
for $\mathcal{A}_{i_1j_1\dots i_Nj_N}=(\textbf{A}_1)_{i_1j_1}\dots (\textbf{A}_N)_{i_Nj_N}$. By the definitions of outer product and the Einstein product, it follows that $ \mathcal{Y}=\mathcal{A}*\mathcal{X}$ where $ \mathcal{A}=\textbf{A}_1\circ \textbf{A}_2\circ \dots \circ \textbf{A}_N$. Hence, the result follows immediately.
\end{proof}

The Einstein product representation (\ref{eq11}) of MLTI systems is indeed the generalization of (\ref{eq9}), and overcomes most of the limitations of Tucker product based MTLI representation discussed above. More importantly, it takes a form similar to the standard LTI system model, and so the representation is more natural for developing MLTI systems theory which we discuss next.

\subsection{Solution of MLTI System}
We first investigate the elementary solution to MLTI systems (\ref{eq11}), which is crucial in the analysis of stability, reachability and observability.
\begin{proposition}\label{pro9}
For the unforced MLTI system
\begin{equation}
\mathcal{X}_{t+1}=\mathcal{A}*\mathcal{X}_t,\label{eq49}
\end{equation}
the solution for $\mathcal{X}$ at time $k$ given initial condition $\mathcal{X}_0$ is $\mathcal{X}_k=\mathcal{A}^{*k}* \mathcal{X}_0$ where $\mathcal{A}^{*k}=\underbrace{\mathcal{A}*\mathcal{A}*\dots*\mathcal{A}}_{k}$.
\end{proposition}

The proof is straightforward using the notion of even-order paired tensors and the Einstein product. If the even-order paired tensor is of the form $\mathcal{A}=\textbf{A}_1\circ \textbf{A}_2\circ \dots \circ \textbf{A}_N$, the $k$-th power Einstein product of $\mathcal{A}$ can be computed by $\mathcal{A}^{*k}=\textbf{A}_1^k\circ \textbf{A}_2^k\circ \dots \circ \textbf{A}_N^k$. Applying Proposition \ref{pro9}, we can write down the explicit solution of (\ref{eq11}) which takes an analogous form to the LTI system,
\begin{equation}
\mathcal{X}_k=\mathcal{A}^{*k}* \mathcal{X}_0+\sum_{j=0}^{k-1}\mathcal{A}^{*k-j-1}*\mathcal{B}*\mathcal{U}_j. \label{eq:31}
\end{equation}

\subsection{Stability.} %Stability theory plays an important role in systems engineering, especially in the study of control systems regarding to both dynamics and control \cite{Chen_2004}.
There are many notions of stability for dynamical systems \cite{Brockett2015,Rugh_1996,Kailath_1980}. For LTI systems, it is conventional to investigate the so-called internal stability. Generalizing from LTI systems, the equilibrium point $\mathcal{X}=\mathcal{O}$ of an unforced MLTI system is called stable if $\|\mathcal{X}_t\|\leq \gamma \|\mathcal{X}_0\|$ for some $\gamma>0$, asymptotically stable if $\|\mathcal{X}_t\| \rightarrow 0$ as $t\rightarrow \infty$, and unstable if it is not stable. %\cite{Scruggs_2017}.

\begin{proposition}
For an unforced MLTI system (\ref{eq49}), the equilibrium point $\mathcal{X}=\mathcal{O}$ is:
\begin{itemize}
\item stable if and only if the magnitudes of all the U-eigenvalues of $\mathcal{A}$ are less than or equal to 1; for those equal to 1, its algebraic and geometry multiplicity must be equal;
\item asymptotically (or exponentially) stable if the magnitudes of all the U-eigenvalues are less than 1;
\item unstable if the magnitudes of some of the U-eigenvalues are greater than 1.
\end{itemize}
\end{proposition}
\begin{proof}
%The definitions of algebraic and geometry multiplicity are defined analogously as in vector spaces.
We only focus on the case when  $\mathcal{A}$ has a full set of U-eigentensors, i.e. $\mathcal{A}=\mathcal{V}*\mathcal{D}*\mathcal{V}^{-1}$ in (\ref{eq216}). It follows from the solution of system (\ref{eq49}) that
\begin{align*}
\mathcal{A}^{*k} = \sum_{\textbf{j}=\textbf{1}}^{\mathscr{J}} \lambda^k_{j_1j_1\dots j_Nj_N} \mathcal{W}_{j_1j_1\dots j_Nj_N},
\end{align*}
where, $\lambda_{j_1j_1\dots j_Nj_N}$ are the U-eigenvalues of $\mathcal{A}$, and $\mathcal{W}_{j_1j_1\dots j_Nj_N}$ are some even-order paired tensors. Then the results follow immediately.
\end{proof}

\subsection{Reachability}
%Reachability and observability are two of the most important concepts in modern control theory, and they are closely related to feedback control, observer design, quadratic regulator, pole assignment, realization theory, and so forth \cite{Rugh_1996}\cite{Kailath_1980}.
%In this and the following subsection, we extend the definitions of reachability and observability for MLTI systems which actually are similar to the ordinary concepts in the vector space, and the sufficient and necessary conditions of reachability and observability for MLTI systems are stated based on existing results \cite{Brockett2015}.
Here and in the following subsection, we introduce the definitions of reachability and observability for MLTI systems which are similar to analogous concepts for the LTI systems \cite{Brockett2015,Rugh_1996,Kailath_1980}. We then establish sufficient and necessary conditions for reachability and observability for the MLTI systems.

\begin{Definition}
 The MLTI system (\ref{eq11}) is said to be \textit{reachable} on $[t_0,t_1]$ if, given any initial condition $\mathcal{X}_0$ and any final state $\mathcal{X}_1$, there exists a sequence of inputs $\mathcal{U}_t$ that steers the state of the system from $\mathcal{X}_{t_0}=\mathcal{X}_0$ to $\mathcal{X}_{t_1}=\mathcal{X}_1$.
\end{Definition}

\begin{theorem}\label{thm31}
 The pair $(\mathcal{A},\mathcal{B})$ is reachable on $[t_0,t_1]$ if and only if the \textit{reachability Gramian}
\begin{equation}
\mathcal{W}_c(t_0,t_1)=\sum_{t=t_0}^{t_1-1}\mathcal{A}^{*t_1-t-1}*\mathcal{B}*\mathcal{B}^{\top}*(\mathcal{A}^{\top})^{*t_1-t-1},
\end{equation}
which is a weakly symmetric even-order square tensor, is U-positive definite.
\end{theorem}
\begin{proof}
Suppose $\mathcal{W}_c(t_0,t_1)$ is U-positive definite, and let $\mathcal{X}_0$ be the initial state and $\mathcal{X}_1$ be the desired final state. Choose $\mathcal{U}_t=\mathcal{B}^{\top}*(\mathcal{A}^{\top})^{*t_1-t-1}*\mathcal{W}_c^{-1}(t_0,t_1)*\mathcal{V}$ for some constant tensor $\mathcal{V}$. It follows from the explicit solution of system (\ref{eq11}) that
\begin{align*}
\mathcal{X}_{t_1}&= \mathcal{A}^{*t_1}* \mathcal{X}_0+\sum_{j=0}^{t_1-1}\mathcal{A}^{*t_1-j-1}*\mathcal{B}*\mathcal{U}_t\\
&= \mathcal{A}^{*t_1}* \mathcal{X}_0+\mathcal{W}_c(t_0,t_1)*\mathcal{W}_c^{-1}(t_0,t_1)*\mathcal{V}\\
&= \mathcal{A}^{*t_1} *\mathcal{X}_0+\mathcal{V}\,.
\end{align*}
Take $\mathcal{V}=-\mathcal{A}^{*t_1}* \mathcal{X}_0+\mathcal{X}_1$, we have $\mathcal{X}_{t_1}=\mathcal{X}_1$.

We show the converse by contradiction. Suppose $\mathcal{W}_c(t_0,t_1)$ is not U-positive definite. Then there exists $\mathcal{X}_a\neq \mathcal{O}$ such that $\mathcal{X}_a^\top*\mathcal{A}^{*t_1-t-1}*\mathcal{B}=\mathcal{O}$ for any $t$. Take $\mathcal{X}_1=\mathcal{X}_a+\mathcal{A}^{*t_1}*\mathcal{X}_0$, and it follows that
\begin{equation*}
\mathcal{X}_a+\mathcal{A}^{*t_1}*\mathcal{X}_0=\mathcal{A}^{*t_1}*\mathcal{X}_0+\sum_{j=t_0}^{t_1-1}\mathcal{A}^{*t_1-j-1}*\mathcal{B}*\mathcal{U}_j\,.
\end{equation*}
Multiplying from the left by $\mathcal{X}_a^\top$ yields
\begin{align*}
\mathcal{X}_a^\top*\mathcal{X}_a=\sum_{j=t_0}^{t_1-1}\mathcal{X}_a^\top*\mathcal{A}^{*t_1-j-1}*\mathcal{B}*\mathcal{U}_j=0,
\end{align*}
which implies that $\mathcal{X}_a=\mathcal{O}$, a contradiction. %So $\mathcal{W}_c(t_0,t_1)$ is U-positive definite, and the proof is thus completed.
\end{proof}

The reachability Gramian assesses to what degree each state is affected by an input \cite{doi.org/10.1142/S0218127405012429}. The infinite horizon reachability Gramian can be computed from the tensor Lyapunov equation which is defined by
\begin{equation}
\mathcal{W}_c - \mathcal{A}*\mathcal{W}_c*\mathcal{A}^\top = \mathcal{B}*\mathcal{B}^\top.
\end{equation}
If the pair $(\mathcal{A},\mathcal{B})$ is reachable and all the U-eigenvalues of $\mathcal{A}$ have magnitude less than 1, one can show that there exists a unique weakly symmetric U-positive definite solution $\mathcal{W}_c$. Solving the infinite horizon reachability Gramian from the tensor Lyapunov equation may be computationally intensive, so a tensor version of the Kalman rank condition is also provided.
%In order to apply the Kalman rank condition, we first need to extend the Cayley-Hamilton theorem for even order paired tensors.
%\begin{lemma}\label{lem3}
%If $\mathcal{A} \in\mathbb{R}^{J_1\times J_1\times \dots \times J_N\times J_N}$ is an evne order square tensor, then $\mathcal{A}$ satisfies its own characteristic polynomial $p(\lambda) = \text{det}(\lambda \mathcal{I}-\mathcal{A})$, i.e. $p(\mathcal{A}) = \mathcal{O}$.
%\end{lemma}
%\begin{proof}
%The proof follows immediately by applying the unfolding transformation $\varphi$.
%\end{proof}

\begin{proposition}\label{coro3.1}
The pair $(\mathcal{A},\mathcal{B})$ is reachable if and only if the $J_1\times J_1K_1\times \dots \times J_N\times  J_NK_N$ even-order reachability tensor
\begin{equation}
\mathscr{R}=
\begin{Vmatrix}
\mathcal{B} & \mathcal{A}*\mathcal{B} & \dots & \mathcal{A}^{*|\mathscr{J}|-1}*\mathcal{B}
\end{Vmatrix}
\end{equation}
spans $\mathbb{R}^{J_1\times J_2\times\dots \times J_N}$. In other words, $\text{rank}_U(\mathscr{R})=|\mathscr{J}|$.
\end{proposition}
\begin{proof}
Using the unfolding transformation $\varphi$ one can express, %such that for some permutation matrix $\textbf{P}$,
 \begin{equation*}
 \varphi(\mathscr{R}) = \begin{bmatrix}
\textbf{B} &  \textbf{A}\textbf{B} & \dots &  \textbf{A}^{|\mathscr{J}|-1}\textbf{B}
 \end{bmatrix}\textbf{P},
 \end{equation*}
where, $\textbf{A}=\varphi(\mathcal{A})$, $\textbf{B}=\varphi(\textbf{B})$ and $\textbf{P}$ is some permutation matrix. Then the result follows immediately.
\end{proof}

Alternatively, the proof can be developed directly in the tensor setting by applying Lemma \ref{lem3} like the approach used in Theorem \ref{thm31}. When $N=1$, Corollary \ref{coro3.1} simplifies to the famous Kalman rank condition for reachability of LTI systems.

\subsection{Observability}
\begin{Definition}
The MLTI system (\ref{eq9}) is said to be \textit{observable} on $[t_0,t_1]$ if any initial state $\mathcal{X}_{t_0}=\mathcal{X}_0$ can be uniquely determined by $\mathcal{Y}_t$ on $[t_0,t_1]$.
\end{Definition}

\begin{theorem}
The pair $(\mathcal{A},\mathcal{C})$ is observable on $[t_0,t_1]$ if and only if the observability Gramian
\begin{equation}
\mathcal{W}_o(t_0,t_1)=\sum_{t=t_0}^{t_1-1}(\mathcal{A}^{\top})^{*t-t_0}*\mathcal{C}^{\top}*\mathcal{C}*\mathcal{A}^{*t-t_0}\,,
\end{equation}
which is a weakly symmetric even-order square tensor, is U-positive definite.
\end{theorem}

The observability Gramian assesses to what degree each state affects future outputs \cite{doi.org/10.1142/S0218127405012429}. The infinite horizon observability Gramian can be computed from the tensor Lyapunov equation defined by
\begin{equation}
\mathcal{A}^\top*\mathcal{W}_o*\mathcal{A} - \mathcal{W}_o = -\mathcal{C}^\top*\mathcal{C}\,.
\end{equation}
If the pair $(\mathcal{A},\mathcal{C})$ is observable and all the U-eigenvalues of $\mathcal{A}$ have magnitude less than 1, there exists a unique weakly symmetric U-positive definite solution $\mathcal{W}_o$.

\begin{proposition}
The pair $(\mathcal{A},\mathcal{C})$ is observable if and only if the $I_1J_1\times J_1\times \dots \times I_NJ_N\times J_N$ even-order observability tensor
\begin{equation}
\mathscr{O}=
\begin{Vmatrix}
\mathcal{C}\\
 \mathcal{C}*\mathcal{A}\\
 \vdots\\
 \mathcal{C}*\mathcal{A}^{*|\mathscr{J}|-1}
\end{Vmatrix}
\end{equation}
spans $\mathbb{R}^{J_1\times J_2\times\dots \times J_N}$. In other words, $\text{rank}_U(\mathscr{O})=|\mathscr{J}|$.
\end{proposition}

\section{Numerical Example}\label{sec:example}
To illustrate MLTI systems theory, we consider a simple single input single output system that is given by (\ref{eq9}) with
\begin{align*}
\textbf{A}_1&=\begin{bmatrix} 0 & 1 & 0\\ 0 & 0 & 1\\0.2 & 0.5 & 0.8\end{bmatrix}, \text{ }\textbf{A}_2=\begin{bmatrix} 0 & 1\\0.5 & 0\end{bmatrix},\\
\textbf{B}_1 &= \begin{bmatrix} 0\\0\\1\end{bmatrix},\text{ } \textbf{B}_2 = \begin{bmatrix} 0 \\1 \end{bmatrix},\\
\textbf{C}_1 &= \begin{bmatrix} 1 & 0 & 0 \end{bmatrix}, \text{ }\textbf{C}_2 = \begin{bmatrix} 1 & 0 \end{bmatrix},
\end{align*}
and the states $\mathcal{X}_t \in \mathbb{R}^{3\times 2}$ are second-order tensors, i.e. matrices. The U-eigenvalues of $\mathcal{A}=\textbf{A}_1\circ \textbf{A}_2$ are $\pm 0.9207$, $- 0.1775\pm 0.2128i$ and $0.1775\pm 0.2128i$. Hence, the MLTI system is asymptotically stable. In addition, the reachability and observability tensors are given by
\begin{align*}
\mathscr{R}_{::11} &=
\begin{bmatrix}
0 & 0 & 0\\
0 & 1 & 0\\
0 & 0.8 & 0
\end{bmatrix}, \hspace{1.1cm}
\mathscr{R}_{::21} =
\begin{bmatrix}
0 & 0 & 0.5\\
0 & 0 & 0.4\\
1 & 0 & 0.57
\end{bmatrix},\\
\mathscr{R}_{::12} &=
\begin{bmatrix}
0.4 & 0 & 0.378\\
0.57 & 0 & 0.4849\\
0.756 & 0 & 0.6339
\end{bmatrix},
\mathscr{R}_{::22} =
\begin{bmatrix}
0 & 0.285 & 0\\
0 & 0.378 & 0\\
0 & 0.4849 & 0
\end{bmatrix},\\
\end{align*}
and
\begin{align*}
\mathscr{O}_{::11} &=
\begin{bmatrix}
1 & 0 & 0\\
0 & 0 & 0\\
0 & 0 & 0.5
\end{bmatrix}, 
\mathscr{O}_{::21} =
\begin{bmatrix}
0 & 0 & 0\\
0.04 & 0.15 & 0.285\\
0 & 0 & 0
\end{bmatrix},\\
\mathscr{O}_{::12} &=
\begin{bmatrix}
0 & 0 & 0\\
0 & 1 & 0\\
0 & 0 & 0
\end{bmatrix},\hspace{0.3cm}
\mathscr{O}_{::22} =
\begin{bmatrix}
0.1 & 0.25 & 0.4\\
0 & 0 & 0\\
0.057 & 0.1825 & 0.378
\end{bmatrix}.\\
\end{align*}
Furthermore, $\text{rank}_U(\mathscr{R})=6$ and $\text{rank}_U(\mathscr{O})=6$, and the system therefore is both reachable and observable. For all the computations in this example, we used the unfolding transform $\varphi$ which enabled us to use standard matrix algebra.

\emph{Remark:} Note that unfolding transform allows one to transform tensor algebra problems to standard matrix algebra problems. However, it may not be the most memory and numerically efficient approach. In fact, computing tensor algebraic notions without unfolding is an active area of research \cite{doi:10.1137/100804577,tensortrain,doi:10.1080/03081087.2018.1500993,Huang_2018}, and we are currently exploring methods based on that for computations associated with the MLTI systems.

\section{Conclusion} \label{sec:future}
In this paper, we generalized the  MLTI system representation using even-order paired tensors and the Einstein product.  This new representation also facilitated the generalization of notions of  stability, reachability and observability from classical multivariate control theory to that for MLTI systems. In particular, the unfolding isomorphism played a key role in establishing criterion for stability, reachability and observability for MLTI systems.

In future work, it should be worthwhile to develop an associated theoretical and computational framework for data driven model identification/reduction, observer  and feedback control design, and to apply these techniques to real world engineering systems and machine learning. One particular application we plan to investigate is that of cellular reprogramming which involves introducing transcription factors as a control mechanism to transform one cell type to another. 
These systems naturally have matrix or tensor state spaces describing their
genome-wide structure and gene expression \cite{Ronquist11832,LIU2018}. Such applications would also ideally be analyzed using  nonlinearity and stochasticity in the multiway dynamical system representation and analysis framework. This is an important direction for future research.

\smallskip
\noindent \textbf{Acknowledgments}: We thank Dr. Frederick Leve at the Air Force Office of Scientific Research (AFOSR) for support and encouragement. This work is supported in part under AFOSR Award No: FA9550-18-1-0028, NSF grant DMS 1613819 and the Lifelong Learning Machines program from DARPA/MTO.

\end{document}